\documentclass[11pt]{amsart}
\usepackage{amssymb, latexsym}

\newcommand{\beqa}{\begin{eqnarray*}}
\newcommand{\eeqa}{\end{eqnarray*}}
\newcommand{\beqn}{\begin{eqnarray}}
\newcommand{\eeqn}{\end{eqnarray}}

\newcommand{\iy}{\infty}

\newcommand{\R}{\mathbb R}

\newcommand{\N}{\mathbb N}

\newcommand{\mcH}{\mathcal H}
\newcommand{\mcB}{\mathcal B}

\newcommand{\al}{\alpha}

\newcommand{\e}{\varepsilon}

\newcommand{\de}{\delta}

\newcounter{cnt1}
\newcounter{cnt2}
\newcounter{cnt3}
\newcommand{\blr}{\begin{list}{$($\roman{cnt1}$)$}
 {\usecounter{cnt1} \setlength{\topsep}{0pt}
 \setlength{\itemsep}{0pt}}}
\newcommand{\bla}{\begin{list}{$($\alph{cnt2}$)$}
 {\usecounter{cnt2} \setlength{\topsep}{0pt}
 \setlength{\itemsep}{0pt}}}
\newcommand{\bln}{\begin{list}{$($\arabic{cnt3}$)$}
 {\usecounter{cnt3} \setlength{\topsep}{0pt}
 \setlength{\itemsep}{0pt}}}
\newcommand{\el}{\end{list}}
\newtheorem{thm}{Theorem}[section]

\newtheorem{ex}[thm]{Example}

\newtheorem{Def}[thm]{Definition}

\newtheorem{rem}[thm]{Remark}
\newcommand{\Rem}{\begin{rem} \rm}
\newcommand{\bdfn}{\begin{Def} \rm}
\newcommand{\edfn}{\end{Def}}

\newcommand{\ba}{\begin{array}}
\newcommand{\ea}{\end{array}}

\numberwithin{equation}{section}
\date{}
\begin{document}
\title{\bf{The M-basis Problem for Separable Banach Spaces}}
\author[Gill]{T. L. Gill}
\address[Tepper L. Gill]{ Departments of Mathematics, Physics, and Electrical \& Computer Engineering, Howard University\\
Washington DC 20059 \\ USA, {\it E-mail~:} {\tt tgill@howard.edu}}
%\abstract{ write abstract here}
\date{}
%thispagestyle{empty}
\subjclass{Primary (46B03), (47D03) Secondary(47H06), (47F05) (35Q80)}
\keywords{spectral theorem, vector measures,  vector-valued functions, Reflexive Banach spaces}
\maketitle
\begin{abstract}    In this note we show that, if $\mcB$ is separable Banach space, then there is a biorthogonal system $\{x_n, x_n^*\}$ such that, the closed linear span of $\{x_n\},\overline{\left\langle {\{x_n\} }\right\rangle}=\mcB$ and $\left\| {x_n } \right\|\left\| {x_n^* } \right\| = 1$ for all  $n$.  
\end{abstract}
\section{Introduction}
In 1943, Marcinkiewicz \cite{M} showed that every separable Banach space  $\mcB$ has a biorthogonal system  $\{x_n, x_n^*\}$  and $\overline{\left\langle {\{x_n\} }\right\rangle}=\mcB$.  This biorthogonal system is now known as a  Marcinkiewicz basis for $\mcB$. 
A well-known open problem is whether one can choose the system $\{x_n, x_n^*\}$ such that $\left\| {x_n } \right\|\left\| {x_n^* } \right\| = 1$ (see Diestel \cite{D}). This  is called the M-basis problem for separable Banach spaces.  The problem has been studied by Singer \cite{SI}, Davis and Johnson \cite{DJ}, Ovsepian and Pelczy\'{n}iski \cite{OP},  Pelczy\'{n}iski \cite{PE} and Plichko \cite{PL}.  The work of Ovsepian and Pelczy\'{n}iski \cite{OP} led to the construction of a bounded M-basis, while  that of Pelczy\'{n}iski \cite{PE} and Plichko \cite{PL} led to independent proofs that, for every  $\e>0$, it is possible to find a biorthogonal system with the property that $\left\| {x_n } \right\|\left\| {x_n^* } \right\| < 1+\e$.   The purpose of this note is to show that we can find a biorthogonal system with the property that $\left\| {x_n } \right\|\left\| {x_n^* } \right\| =1$.\subsection{Preliminaries}
\begin{Def}Let $\mcB$ be  an infinite-dimensional separable Banach space.  The family $ \{x_i \}_{i=1}^{\infty} \subset \mcB$ is called:
\begin{enumerate}
\item A \emph{fundamental} system if $\overline{\left\langle { \{x_j : j \in \mathbb{N} \} }\right\rangle}=\mcB$.
\item A \emph{minimal} system if $x_i \notin \overline{\left\langle { \{x_j : j \in \mathbb{N} \setminus\{i\} \} }\right\rangle}$.
\item A \emph{total} if for each $x\ne 0$ there exists $i\in \mathbb{N}$ such that $x_i^{*}(x)\ne 0$.
\item A \emph{biorthogonal} system if $x_i^*(x_j)=\de_{ij}$, for all $i,j\in \mathbb{N}$.
\item A \emph{Markushevich basis} (or M-basis) if it is a fundamental minimal and total biorthogonal sequence. 
\end{enumerate}
\end{Def}
To understand the M-basis problem and its solution in a well-known setting, let $\R^2$ have its standard inner product $(\,  \cdot, \cdot \,)$ and let $x_1, \;x_2$ be any two independent basis vectors,.     Define a new inner product on $\R^2$ by 
\beqn
\begin{gathered}
  \left\langle {y}
 \mathrel{\left | {\vphantom {y z}}
 \right. \kern-\nulldelimiterspace}
 {z} \right\rangle  = {t_1}\left( {x_1^{} \otimes x_1^{}} \right)\left( {y \otimes z} \right) + {t_2}\left( {x_2^{} \otimes x_2^{}} \right)\left( {y \otimes z} \right) \hfill \\
  \quad \quad  = {t_1}\left( {y,x_1^{}} \right)\left( {z,x_1^{}} \right) + {t_2}\left( {y,x_2^{}} \right) \left( {z,x_2^{}} \right), \hfill \\ 
\end{gathered} 
\eeqn
where $t_1, \, t_2>0, \; t_1+t_2=1$.  Define new functionals $S_1$ and $S_2$ by:

\[
{S_1}(x) = \frac{{\left\langle {x}
 \mathrel{\left | {\vphantom {x {{x_1}}}}
 \right. \kern-\nulldelimiterspace}
 {{{x_1}}} \right\rangle }}{{{\alpha _1}\left\langle {{{x_1}}}
 \mathrel{\left | {\vphantom {{{x_1}} {{x_1}}}}
 \right. \kern-\nulldelimiterspace}
 {{{x_1}}} \right\rangle }},\quad {S_2}(x) = \frac{{\left\langle {x}
 \mathrel{\left | {\vphantom {x {{x_2}}}}
 \right. \kern-\nulldelimiterspace}
 {{{x_2}}} \right\rangle }}{{{\alpha _2}\left\langle {{{x_2}}}
 \mathrel{\left | {\vphantom {{{x_2}} {{x_2}}}}
 \right. \kern-\nulldelimiterspace}
 {{{x_2}}} \right\rangle }},\quad {\text{for}}\quad y \in {\mathbb{R}^2}.
\]
Where $\al_1, \; \al_2 >0$ are chosen to ensure that $\left\| {{S_1}} \right\| = \left\| {{S_2}} \right\| = 1$.
Note that, if $(x_1, x_2)=0$, $S_1$ and $S_2$ reduce to 
\[
{S_1}(x) = \frac{{\left( {x,x_1^{}} \right)}}{{{\alpha _1}\left\| {{x_1}} \right\|}},\quad {S_2}(x) = \frac{{\left( {x,x_2^{}} \right)}}{{{\alpha _2}\left\| {{x_2}} \right\|}}.
\]
Thus, we can define many equivalent inner products on $\R^2$ and many linear functionals with the same properties but different norms. 

The following example shows how this construction can be of use. 
\begin{ex}
In this example, let $x_1=e_1$ and $x_2=e_1+e_2$, where $e_1=(1,0), \; e_2=(0,1)$. In this case, the biorthogonal functionals are generated by the vectors $\bar{x}_1=e_1-e_2$ and $\bar{x}_2=e_2$ {\rm{(}{i.e.}, $x_1^*(x) = \left( {x,{{\bar x}_1}} \right),\quad x_2^*(x) = \left( {x,{{\bar x}_2}} \right)$\rm{)}}.  It follows that $(x_1,\bar{x}_2)=0, \; (x_1,\bar{x}_1)=1$ and $(x_2, \bar{x}_1)=0, \; (x_2, \bar{x}_2)=1$.  However,  $\left\| {x_1 } \right\|\left\| {\bar{x}_1 } \right\| = \sqrt 2 ,\quad \left\| {x_2 } \right\|\left\| {\bar{x}_2 } \right\| = \sqrt 2$, so that 
$\left\{ {{x_1},\left( {\; \cdot ,\;{{\bar x}_1}} \right)} \right\}$ and $\left\{ {{x_2},\left( {\; \cdot ,\;{{\bar x}_2}} \right)} \right\}$ fails to solve the M-basis problem on $\R^2$.
 
In this case, we  set $\al_1=1$ and $\al_2= \left\| {{{x}_2}} \right\|$ so that, without changing $x_1$ and $x_2$, and using the inner product from equation (1.1) in the form 
\[
  \left\langle {x}
 \mathrel{\left | {\vphantom {x y}}
 \right. \kern-\nulldelimiterspace}
 {y} \right\rangle  = {t_1}\left( {x,\bar{x}_1^{}} \right)\left( {y,\bar{x}_1^{}} \right) + {t_2}\left( {x,\bar{x}_2^{}} \right) \left( {y,\bar{x}_2^{}} \right), 
\]
$S_1$ and $S_2$ become
\[
{S_1}(x) = \frac{{\left( {x,\bar{x}_1^{}} \right)}}{{\left\| {{\bar{x}_1}} \right\|}},\quad {S_2}(x) = \frac{{\left( {x,\bar{x}_2^{}} \right)}}{{\left\| {{{x}_2}} \right\|}}.
\]
It now follows that $S_i(x_i)=1$ and $S_i(x_j)=0$ for $i \ne j$ and    $\left\| {S_i^{} } \right\|\left\| {x_i^{} } \right\| = 1$,  so that system $\{ {x_1, S_1 } \}$ and $\{ {x_2, S_2 } \}$  solves the M-basis problem.
\end{ex}
\begin{rem} For a given set of independent vectors on a finite dimensional vector space, It is known that the corresponding biorthogonal functionals are unique.  This example shows that uniqueness is only up to a scale factor and this is what we need to produce an M-basis. 
\end{rem}
The following theorem is one of the main ingredients in our solution to the general M-basis problem.  (It is a variation of a result due to Kuelbs \cite{KB}.)
\begin{thm} Suppose ${\mathcal{B}}$ is a separable  Banach space, then there exist a separable Hilbert space ${\mathcal{H}}$  such that, ${\mathcal{B}} \subset {\mathcal{H}}$ as a continuous dense embedding.
\end{thm}
\begin{proof} Let $\{ u_n \} $ be a countable dense sequence in ${\mathcal{B}}$ and let $\{ u_n^* \} $ be any fixed set of corresponding duality mappings (i.e., $u_n^*  \in {\mathcal{B}^*}$, the dual space of ${\mathcal{B}}$ and $
u_n^* (u_n ) = \left\langle {u_n ,u_n^* } \right\rangle  = \left\| {u_n } \right\|_{\mathcal{B}}^2 = \left\| {u_n^* } \right\|_{\mathcal{B}^*}^2 =1$).   Let $\{ t_n \}$ be a positive sequence of numbers such that $\sum\nolimits_{n = 1}^\infty  {t_n }  = 1$, and define an inner product on $\mcB$, $\left( {u,v} \right)$, by:
\[
\left( {u,v} \right) = \sum\nolimits_{n = 1}^\infty  {t_n u_n^* (u)} \bar{u}_n^* (v).
\]
It is easy to see that $\left( {u,v} \right)$ is an inner product on ${\mathcal{B}}$.  Let $
{\mathcal{H}}$ be the completion of ${\mathcal{B}}$ with respect to this inner product.  It is clear that ${\mathcal{B}}$ is dense in ${\mathcal{H}}$, and 
\[
\left\| u \right\|_\mcH^2  = \sum\nolimits_{n = 1}^\infty  {t_n \left| {u_n^* (u)} \right|^2 }  \le \sup _n \left| {u_n^* (u)} \right|^2  \le \left\| u \right\|_{\mathcal{B}}^2,
\]
so the embedding is continuous.
\end{proof}
\subsection{The Main Result}
The following theorem shows how our solution to the M-basis problem for $\R^2$ can be extended to any separable Banach space.

\begin{thm} Let $\mcB$ be a infinite-dimensional separable Banach space.  Then $\mcB$ contains  an  M-basis with the property that $\left\| {x_i } \right\|_{\mcB} \left\| {x_i^* } \right\|_{{\mcB}^*} = 1$ for all $i$. 
\end{thm}
\begin{proof}
Construct $\mcH$ via Theorem 1.3, so that  $\mcB \subset \mcH$ is a dense continuous embedding and let ${\left\{ {x_i } \right\}_{i=1}^\iy}$ be a fundamental minimal system for $\mcB$.   If $i \in \N$, let $M_{i, \mcH}$ be the closure of the span of ${\left\{ {x_j } \right\}}$ in $\mcH$, where $j \ne i$  (i.e., $M_{i, \mcH} = \overline {\left\langle {\left\{ {x_j } \right\}\left| {j \in } \right.\mathbb{N},j \ne i} \right\rangle }$).  Thus, $x_i \in M_{i, \mcH}^\bot$, $M_{i,\mcH}  \oplus M_{i,\mcH}^ \bot   = \mcH$ and $\left( {y,x_i } \right)_\mcH = 0$ for all $y \in M_{i, \mcH}$.  

Let $M_{i}$ be the closure of the span of ${\left\{ {x_j \; j \ne i} \right\}}$ in $\mcB$.  Since $M_{i} \subset M_{i, \mcH}$ and $x_i \notin M_i,  \; (y, x_i)_\mcH=0$ for all $y \in M_i$.  Let the seminorm $p_i(\, \cdot \,)$ be defined on the closure of the span of $\{x_i \}, \; \overline {\left\langle {\left\{ {x_i } \right\}} \right\rangle }$ by $p_i(y) = \left\| {x_i } \right\|_\mcB \left\| y \right\|_\mcB$,
and define $ {\hat x}_i^* (\, \cdot \,)$ by:
\[
 {\hat x}_i^* (y) = \frac{{\left\| x_i \right\|_{\mcB}^2 }}
{{\left\| x_i \right\|_{\mcH}^2 }}\left( {y, x_i} \right)_{\mcH}. 
\]
By the Hahn-Banach Theorem, ${\hat x}_i^* (\, \cdot \,)$ has an extension $x_i^* (\, \cdot \,)$ to $\mcB$, such that $\left| {x_i^* (y)} \right|   \leqslant p_i (y)= \left\| x_i \right\|_B \left\| y \right\|_B$ for all $y \in \mcB$. By definition of $p_i(\, \cdot \,)$, we see that $\left\| {x_i^* } \right\|_{\mathcal{B}^*} \le \left\| x_i \right\|_\mathcal{B}$.   On the other hand  $x_i^* (x_i) = \left\| x_i \right\|_\mathcal{B}^2 \leqslant \left\| x_i \right\|_\mathcal{B} \left\| {x_i^* } \right\|_{\mathcal{B}^*}$, 
so that $x_i^* ( \, \cdot \, )$ is a duality mapping for $x_i$.  If $x_i^*(x)=0$ for all $i$, then $x \in \bigcap_{i=1}^\iy{M_i} =\{0\}$ so that the family ${\left\{ {x_i^* } \right\}_{i=1}^\iy}$ is total.  If  we let $\left\| x_i \right\|_\mathcal{B}=1$, it is clear that $x_i^*(x_j)= \de_{ij}$, for all $i,j \in \N$. Thus, $\{x_i, x_i^* \}$ is an  M-basis system with $\left\| {x_i } \right\|_{\mcB} \left\| {x_i^* } \right\|_{{\mcB}^*} = 1$ for all $i$.
\end{proof}

\end{document}